\documentclass[reqno]{amsart}
\usepackage[utf8]{inputenc}
\usepackage[T1]{fontenc}
\usepackage{lmodern}
\usepackage{amsmath}
\usepackage{amssymb}
\usepackage{array}
\usepackage{xcolor}
\usepackage{xspace}
\usepackage{tikz}
\usepackage{graphicx}
\usepackage{booktabs}
\usepackage{enumerate}

\newtheorem{theo}{Theorem}
\newtheorem{lemm}{Lemma} 
\newtheorem{prop}{Proposition}
\newtheorem{coro}{Corollary}

\newcommand{\pqs}{\mbox{$(p,q)$-\textsc{scp}}}
\newcommand{\pqsp}[2]{\mbox{$(#1,#2)$-\textsc{scp}}}
\newcommand{\pqc}{\mbox{$(p,q)$-ary chain}}
\newcommand{\n}{\mathbb{N}}
\newcommand{\eps}{\varepsilon}

\tikzstyle{grid}=[dotted, very thin]
\tikzstyle{walk}=[dashed]

\definecolor{myblue}{rgb}{0,0.1,0.9}
\definecolor{mygreen}{rgb}{0, 0.65, 0.2}
\definecolor{myred}{rgb}{0.8,0,0.2}

\title[On $(P,Q)$-ary chain partitions with bounded parts]{On the maximal weight
  of $(P,Q)$-ary chain partitions with bounded parts}
\author{Filippo Disanto}
\address{Institut für Genetik, Universität zu Köln, Germany}
\email[F.~Disanto]{disafili@yahoo.it}

\author{Laurent Imbert}
\address{LIRMM, CNRS, Université Montpellier 2, France}
\email[L.~Imbert]{Laurent.Imbert@lirmm.fr}

\author{Fabrice Philippe} 
\address{LIRMM, CNRS, Université Montpellier 2, Université Montpellier 3, France}
\email[F.~Philippe]{Fabrice.Philippe@lirmm.fr}
\date{\today}

\begin{document}

\begin{abstract}
  A $(p,q)$-ary chain is a special type of chain partition of integers with
  parts of the form $p^aq^b$ for some fixed integers $p$ and $q$.  In this note,
  we are interested in the maximal weight of such partitions when their parts
  are distinct and cannot exceed a given bound $m$.  Characterizing the cases
  where the greedy choice fails, we prove that this maximal weight is, as a
  function of $m$, asymptotically independent of $\max(p,q)$, and we
  provide an efficient algorithm to compute it.
\end{abstract}

\maketitle

\section{Introduction}
\label{sec:introduction}

Let $p,q$ be two fixed integers, and let $E=\{p^aq^b:(a,b)\in\n^2\}$ be endowed
with the divisibility order, i.e., $x \succeq y \iff y \mid x$.  A \pqc\ is a
finite non-increasing sequence in $E$.  For example, $(72, 12, 4, 4,1)$ is a
$(2,3)$-ary chain, whereas $(72, 12, 4, 3, 1)$ is not since $4 \not\succeq 3$.
We define the \emph{weight} of a \pqc\ as the sum of its terms:
\begin{equation}
  \label{eq:weight}
  w = \sum_{i\geq 1} p^{a_i}q^{b_i}, \; \text{ where }
  p^{a_i}q^{b_i} \succeq p^{a_{i+1}}q^{b_{i+1}} \text{ for } i \geq 1
\end{equation}

Expansions of this type have been proposed and successfully used by
Dimitrov~et~al. in the context of digital signal processing and cryptography
under the name \emph{double-base number system}. (For more details
see~\cite{DimJulMil99,DimImbMis08:mathcomp} and the references therein.)

From a different point of view, a \pqc\ can be seen as a partition of its
weight, where the parts are restricted to the set $E$ and constrained by a
divisibility condition.  Surprisingly, works on integer partitions with
divisibility constraints on the parts are very scarce.  Erd\H{o}s and Loxton
considered two types of such unconventional partitions, called chain and
umbrella partitions~\cite{ErdLox79a}, and obtained ``\emph{some rather weak
  estimates for various partition functions}''.  More recently, motivated by
some theoretical questions behind Dimitrov's number system, the second and third
authors refined some of Erd\H{o}s and Loxton's earlier results in a paper
entitled \emph{strictly chained $(p,q)$-ary
  partitions}~\cite{ImbPhi10:pqary:CDM}.  A strictly chained $(p,q)$-ary
partition, or \pqs{} for short, is a decreasing \pqc , i.e. it has distinct
parts.  The original motivation for the present work was to extend the results
from~\cite{ImbPhi10:pqary:CDM} to the unconventional situation where the parts
of a \pqs{} can be either positive or negative. The results of such a study are
expected to provide significant improvements for some cryptographic primitives,
e.g. the computation of the multiple of a point on an elliptic curve.  In
this context the first, natural question that we tackle in the present paper is:
``What is the maximal weight of a \pqs{} whose parts are bounded by some given
integer $m$?''  Although the problem may seem elementary at first glance, we
show that the answer is not so trivial.  In particular, assuming $p<q$, we prove
that this maximal weight asymptotically grows as $pm/(p-1)$, independently of
$q$.

If the first part is given, the heaviest \pqs{} may be computed using a greedy
strategy by successively taking the next greatest part satisfying the
divisibility condition.  Nevertheless, given a bound $m > 0$ on the parts,
determining how to best select the first part is not immediate and the greedy
approach fails in general.  Indeed, we shall see that choosing the largest part
less than or equal to $m$ does not always provide a partition of maximal weight.
These facts are established in Sections~\ref{sec:preliminaries}
and~\ref{sec:maxim-elem} among other preliminary definitions, examples, and
results.  The cases where the greedy choice fails are fully characterized in
Section~\ref{sec:Ym}. Section~\ref{sec:asymptotics} is devoted to the asymptotic
behavior of the maximal weight as a function of $m$.  Finally, in
Section~\ref{sec:computing} we show how to compute a best choice for the first
part, thus the maximal weight, in $O(\log\log m)$ steps.

\section{Preliminaries}
\label{sec:preliminaries}

Let $m$ be a positive integer, and let $G(m)$ denote the maximal weight of a
\pqs{} whose greatest part does not exceed $m$.  For example, with $p=2$ and
$q=3$, the first values of $G$ are: 1, 3, 4, 7, 7, 10, 10, 15, 15, 15, 15, 22,
22, 22, 22, 31, 31, \dots

In the following, we shall assume w.l.o.g. that $p<q$.  Notice that the case
$p=1$ is irrelevant since $G(m)$ is simply the sum of all the powers of $q$ less
than or equal to $m$.  More generally, and for the same reason, we shall
consider that $p$ and $q$ are not powers of the same integer, or equivalently
\emph{multiplicatively independent}. As a direct consequence, both $\log_p{q}$
and $\log_q{p}$ are irrational numbers (see e.g.~\cite[Theorem~2.5.7]{AllSha03}).  Under
this assumption, the first values of $G(m)$ may be quickly computed with the
help of the following formula.
\begin{prop}
  \label{lemm:Gamma(k)}
  For $m \in \n^*$, let $G(m)$ denote the largest integer that can be expressed
  as a strictly chained $(p,q)$-ary partition with all parts less than or equal
  to $m$. Assume that $G(m) = 0$ if $m \not\in \n$. Then, we have $G(1)=1$, and
  for $m>1$
  \begin{equation}
    \label{eq:Gamma}
    G(m) = \max\left( G(m-1), 1+pG(m/p), 1+qG(m/q) \right)
  \end{equation}
\end{prop}

\begin{proof}
  Let $\lambda$ be a partition of weight $G(m)$ whose parts are all less than or
  equal to $m$.  First, notice that $\lambda$ must contain part $1$ by
  definition of $G(m)$.  If $m \not\in E$, then $G(m) = G(m-1)$.  Otherwise, it
  suffices to observe that removing part $1$ from $\lambda$ produces a partition
  whose parts are all divisible by either $p$ or $q$.
\end{proof}

Computing $G(m)$ with relation~\eqref{eq:Gamma} requires $O(\log m)$ steps in
the worst case: Simply note that, for all $m$, in at most $p-1$ baby-steps,
i.e. $G(m) = G(m-1)$, one gets an integer that is divisible by $p$.
Formula~\eqref{eq:Gamma} may also be adapted to compute both $G(m)$ and a
  \pqs{} of such weight.  Nevertheless, it does not give any idea about the
asymptotic behavior of $G$.  Moreover, we shall see in
Section~\ref{sec:computing} how to compute $G(m)$ and a \pqs{} of weight $G(m)$
in $O(\log\log m)$ steps.

A natural graphic representation for \pqs{}s is obtained by mapping each part
$p^aq^b \in E$ to the pair $(a,b) \in \n^2$.  Indeed, with the above assumptions
on $p$ and $q$, the mapping $(a,b) \mapsto p^aq^b$ is one-to-one.  Since the
parts of a \pqs{} are pairwise distinct by definition, this graphic
representation takes the form of an increasing path in $\n^2$ endowed with the
usual product order.  This is illustrated in Figure~\ref{fig:staircase-walk-72}
with the ten \pqsp{2}{3}s containing exactly six parts and whose greatest part
equals $72 = 2^33^2$.  Note that a \pqs{} with largest part $p^aq^b$ possesses at
most $a+b+1$ parts, and that there are exactly $\binom{a+b}{b}$ of them with a
maximum number of parts.
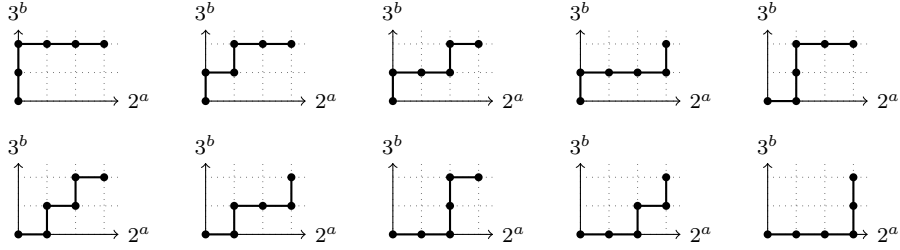
\begin{figure}[h]
  \centering
  \begin{tabular}{ccccc}
    \begin{tikzpicture}[scale=0.38, every text node part/.style={font=\footnotesize}]
      \draw[grid] (0,0) grid (3.5, 2.5);
      \draw[->] (0,0) -- (3.5,0);
      \draw[->] (0,0) -- (0,2.5);
      \draw (3.5,0) node[anchor=west] {$2^a$};
      \draw (0,2.5) node[anchor=south] {$3^b$};
      \draw[thick] (0,0) |- (3,2);
      \foreach \a/\b in {0/0, 0/1, 0/2, 1/2, 2/2, 3/2}
      \fill (\a, \b) circle (4pt);
    \end{tikzpicture}
    &
    \begin{tikzpicture}[scale=0.38, every text node part/.style={font=\footnotesize}]
      \draw[grid] (0,0) grid (3.5, 2.5);
      \draw[->] (0,0) -- (3.5,0);
      \draw[->] (0,0) -- (0,2.5);
      \draw (3.5,0) node[anchor=west] {$2^a$};
      \draw (0,2.5) node[anchor=south] {$3^b$};
      \draw[thick] (0,0) |- (1,1) |- (3,2);
      \foreach \a/\b in {0/0, 0/1, 1/1, 1/2, 2/2, 3/2}
      \fill (\a, \b) circle (4pt);
    \end{tikzpicture}
   & 
    \begin{tikzpicture}[scale=0.38, every text node part/.style={font=\footnotesize}]
      \draw[grid] (0,0) grid (3.5, 2.5);
      \draw[->] (0,0) -- (3.5,0);
      \draw[->] (0,0) -- (0,2.5);
      \draw (3.5,0) node[anchor=west] {$2^a$};
      \draw (0,2.5) node[anchor=south] {$3^b$};
      \draw[thick] (0,0) |- (2,1) |- (3,2);
      \foreach \a/\b in {0/0, 0/1, 1/1, 2/1, 2/2, 3/2}
      \fill (\a, \b) circle (4pt);
    \end{tikzpicture}
    &
    \begin{tikzpicture}[scale=0.38, every text node part/.style={font=\footnotesize}]
      \draw[grid] (0,0) grid (3.5, 2.5);
      \draw[->] (0,0) -- (3.5,0);
      \draw[->] (0,0) -- (0,2.5);
      \draw (3.5,0) node[anchor=west] {$2^a$};
      \draw (0,2.5) node[anchor=south] {$3^b$};
      \draw[thick] (0,0) |- (3,1) |- (3,2);
      \foreach \a/\b in {0/0, 0/1, 1/1, 2/1, 3/1, 3/2}
      \fill (\a, \b) circle (4pt);
    \end{tikzpicture}
    &
    \begin{tikzpicture}[scale=0.38, every text node part/.style={font=\footnotesize}]
      \draw[grid] (0,0) grid (3.5, 2.5);
      \draw[->] (0,0) -- (3.5,0);
      \draw[->] (0,0) -- (0,2.5);
      \draw (3.5,0) node[anchor=west] {$2^a$};
      \draw (0,2.5) node[anchor=south] {$3^b$};
      \draw[thick] (0,0) -| (1,2) -| (3,2);
      \foreach \a/\b in {0/0, 1/0, 1/1, 1/2, 2/2, 3/2}
      \fill (\a, \b) circle (4pt);
    \end{tikzpicture}
    \\
    \begin{tikzpicture}[scale=0.38, every text node part/.style={font=\footnotesize}]
      \draw[grid] (0,0) grid (3.5, 2.5);
      \draw[->] (0,0) -- (3.5,0);
      \draw[->] (0,0) -- (0,2.5);
      \draw (3.5,0) node[anchor=west] {$2^a$};
      \draw (0,2.5) node[anchor=south] {$3^b$};
      \draw[thick] (0,0) -| (1,1) -| (2,2) -| (3,2);
      \foreach \a/\b in {0/0, 1/0, 1/1, 2/1, 2/2, 3/2}
      \fill (\a, \b) circle (4pt);
    \end{tikzpicture}
    &
    \begin{tikzpicture}[scale=0.38, every text node part/.style={font=\footnotesize}]
      \draw[grid] (0,0) grid (3.5, 2.5);
      \draw[->] (0,0) -- (3.5,0);
      \draw[->] (0,0) -- (0,2.5);
      \draw (3.5,0) node[anchor=west] {$2^a$};
      \draw (0,2.5) node[anchor=south] {$3^b$};
      \draw[thick] (0,0) -| (1,1) -| (3,2);
      \foreach \a/\b in {0/0, 1/0, 1/1, 2/1, 3/1, 3/2}
      \fill (\a, \b) circle (4pt);
    \end{tikzpicture}
    &
    \begin{tikzpicture}[scale=0.38, every text node part/.style={font=\footnotesize}]
      \draw[grid] (0,0) grid (3.5, 2.5);
      \draw[->] (0,0) -- (3.5,0);
      \draw[->] (0,0) -- (0,2.5);
      \draw (3.5,0) node[anchor=west] {$2^a$};
      \draw (0,2.5) node[anchor=south] {$3^b$};
      \draw[thick] (0,0) -| (2,2) -| (3,2);
      \foreach \a/\b in {0/0, 1/0, 2/0, 2/1, 2/2, 3/2}
      \fill (\a, \b) circle (4pt);
    \end{tikzpicture}
    &
    \begin{tikzpicture}[scale=0.38, every text node part/.style={font=\footnotesize}]
      \draw[grid] (0,0) grid (3.5, 2.5);
      \draw[->] (0,0) -- (3.5,0);
      \draw[->] (0,0) -- (0,2.5);
      \draw (3.5,0) node[anchor=west] {$2^a$};
      \draw (0,2.5) node[anchor=south] {$3^b$};
      \draw[thick] (0,0) -| (2,1) -| (3,2);
      \foreach \a/\b in {0/0, 1/0, 2/0, 2/1, 3/1, 3/2}
      \fill (\a, \b) circle (4pt);
    \end{tikzpicture}
    &
    \begin{tikzpicture}[scale=0.38, every text node part/.style={font=\footnotesize}]
      \draw[grid] (0,0) grid (3.5, 2.5);
      \draw[->] (0,0) -- (3.5,0);
      \draw[->] (0,0) -- (0,2.5);
      \draw (3.5,0) node[anchor=west] {$2^a$};
      \draw (0,2.5) node[anchor=south] {$3^b$};
      \draw[thick] (0,0) -| (3,2);
      \foreach \a/\b in {0/0, 1/0, 2/0, 3/0, 3/1, 3/2}
      \fill (\a, \b) circle (4pt);
    \end{tikzpicture}
    \\
  \end{tabular}
  \caption{The set of \pqsp{2}{3}s with 6 parts and whose largest part equals
    $2^33^2 = 72$.}
  \label{fig:staircase-walk-72}
\end{figure}

With this representation in mind, one is easily convinced that the heaviest
\pqs{} with first part $p^aq^b$ looks like the top left \pqs{} in
Figure~\ref{fig:staircase-walk-72}.  This is formalized in the following lemma.

\begin{lemm}
  \label{lemm:characterize}
  Given $a,b\in\n$, the heaviest \pqs{} with first part $p^aq^b$ is the one whose
  parts are the elements of the set $\{q^i : 0 \leq i<b\} \cup \{q^bp^i : 0 \leq
  i\leq a\}$.
\end{lemm}

\begin{proof}
  Consider a \pqs{} $\lambda=(\lambda_i)_{i=1}^k$ with greatest part
  $\lambda_1=p^aq^b$.  Let $\lambda_i=p^{a_i}q^{b_i}$.  If $a_i+b_i>b$ then
  define $\lambda'_i=p^{a_i+b_i-b}q^b$, otherwise let $\lambda'_i=q^{a_i+b_i}$.
  Note that $\lambda'_1=p^aq^b$ again.  Since sequence $(a_i+b_i)_{i=1}^k$ is
  decreasing, $(\lambda'_i)_{i=1}^k$ is also a \pqs.  Since $p<q$ we have
  $\lambda'_i\geq\lambda_i$ for all $i$, with equality if, and only if, the
  parts in $\lambda$ form a subset of $\{q^i : 0 \leq i<b\} \cup \{p^iq^b : 0
  \leq i\leq a\}$.  Therefore, the maximal weight is reached when taking the
  whole set, and only in this case.
\end{proof}

As a consequence, a \pqs{} of weight $G(m)$ and whose parts do not exceed $m$ is
characterized by its greatest part only.  Moreover, denoting by $p^aq^b$ this
greatest part, we have $G(m)=h(a,b)$, where $h$ is the mapping defined on $\n^2$
by
\begin{align}\label{def:h}
  h(a,b)=\frac{q^b-1}{q-1}+\frac{p^{a+1}-1}{p-1}q^b
\end{align}
Accordingly, the definition of $G$ may be rewritten as 
\begin{equation}
  \label{def:G}
  G(m) = \max_{P_m} h,\quad \text{where } P_m=\{(a,b)\in\n^2 : p^aq^b \leq m\}
\end{equation}

Finally observe that the greatest part of a \pqs{} of weight $G(m)$ and whose
parts do not exceed $m$ must be a maximal element of $E \cap [0,m]$ for the
divisibility order.  Otherwise, the partition could be augmented by a part,
resulting in a partition of larger weight.  The next section is devoted to the
set of these maximal elements.
 
\section{On the set $Z_m$}
\label{sec:maxim-elem}

For convenience, let us denote by $\rho$ the logarithmic ratio of $q$ and $p$.
$$
\rho = \frac{\log q}{\log p} > 1
$$
Since $p$ and $q$ are multiplicatively independent, $\rho$ is irrational.

Let us further denote by $Z_m$ the set of all maximal elements in $E \cap [0,m]$
for the divisibility order.  Recall that $E=\{p^aq^b:(a,b)\in\n^2\}$, so that any
element of $E$ may also be written as $p^{a+b\rho}$.  There are exactly
$\lfloor\log_q m\rfloor +1$ elements in $Z_m$, described in the following Lemma.

\begin{lemm}
  \label{lemm:Z_m}
  Let $m$ be a positive integer. The following characterization holds:
  $$
  p^aq^b\in Z_m \Longleftrightarrow 0\leq b\leq \lfloor\log_q
  m\rfloor\text{\ and\ }a=\lfloor\log_p m-b\rho\rfloor
  $$
\end{lemm}

\begin{proof}
  An element $p^aq^b$ of $E$ is in $Z_m$ if, and only if, $a$ and $b$ are
  non-negative, $p^aq^b\leq m< p^{a+1}q^b$, and $p^aq^b\leq m< p^{a}q^{b+1}$.
  Since $p<q$, the latter condition is superfluous.  Checking that the former
  inequalities are equivalent to the Lemma's claim is immediate.
\end{proof}

As a consequence, let us note for further use that 
\begin{align}
  Z_{qm} &= qZ_m\cup\{p^{\lfloor \rho + \log_pm\rfloor}\}, \label{eq:Z_qm}\\[1ex]
  Z_{pm} &= 
  \begin{cases}
    pZ_m & \text{if } \lfloor1/\rho+\log_qm\rfloor=\lfloor\log_qm\rfloor ,\\
    pZ_m\cup\{q^{\lfloor\log_qm\rfloor+1}\} & \text{otherwise}
  \end{cases}
\end{align}

The elements of $Z_m$ correspond exactly to the maximal integer points below or
on the line of equation $a\log{p} + b\log{q} - \log{m} = 0$.  An example is
given in Figure~\ref{fig:Z_750}. The corresponding values $p^aq^b$ and $h(a,b)$
are reported in Table~\ref{tab:750}.
\begin{figure}[htbp]
  \centering
  \begin{tikzpicture}[scale=0.65]
    \draw[dotted, very thin] (0,0) grid (13.5,6.5);
    \draw[->, very thin] (0,-0.3) -- (0,6.8) node[anchor=east] {$b$};
    \draw[->, very thin] (-0.3,0) -- (14,0) node[anchor=north] {$a$};
    \draw (-0.3,6.215) -- (10.027,-0.3);
    \foreach \y/\x in {0/9, 1/7, 2/6, 3/4, 4/3, 5/1, 6/0}
    \fill[red] (\x,\y) circle (2pt);
    \draw[dashed, blue] (0,0) -- (0,1) -- (2,2) -- (2,6.5);
    \foreach \y/\x in {0/0, 1/0, 2/2, 3/2, 4/2, 5/2, 6/2}
    {
      \draw[blue, thick] (\x - 0.1, \y - 0.1) -- (\x + 0.1, \y + 0.1);
      \draw[blue, thick] (\x - 0.1, \y + 0.1) -- (\x + 0.1, \y - 0.1);
    }
  \end{tikzpicture}
  \caption{The set $Z_{750}$ for $(p,q)=(2,3)$, represented as all maximal
    integer points below the line of equation $x\log{2} + y\log{3} - \log{750} =
    0$. The points along the dashed line correspond to the first values of the
    sequence $\ell$ defined in Theorem~\ref{th:ell}.}
  \label{fig:Z_750}
\end{figure}
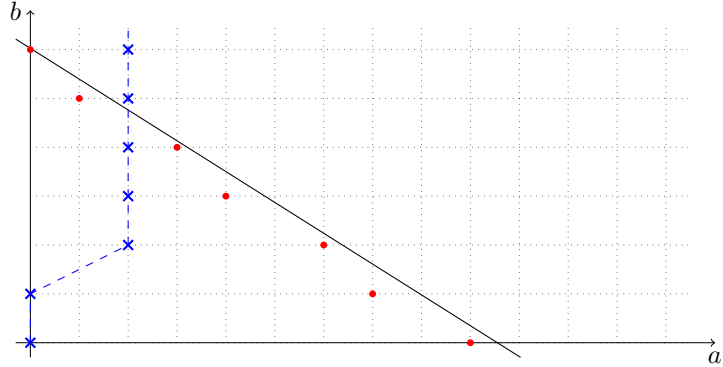

\begin{table}[htbp]
  \centering 
  \caption{The elements of $Z_{750}$ for $(p,q)=(2,3)$, together with the
    corresponding values $p^aq^b$ and $h(a,b)$. Note that $G(750) = h(3,4) =
    1255$.}
  \label{tab:750}
  \begin{tabular}{p{1.25cm}*7{c}}
    $(a,b)$ & (0, 6) & (1, 5) & (3, 4) & (4, 3) & (6, 2) & (7, 1) & (9, 0) \\
    \midrule
    $p^aq^b$ & 729 & 486 & 648 & 432 & 576 & 384 & 512 \\
    $h(a,b)$ & 1093 & 850 & 1255 & 850 & 1147 & 766 & 1023\\
  \end{tabular}
\end{table}

Further define $z_m$ as the greatest integer of the form $p^aq^b$ less than or
equal to $m$, that is,
$$
z_m = \max Z_m 
$$ 
Since $q^{\lfloor \log_q{m}\rfloor}\in Z_m$, we have $z_m \rightarrow \infty$
when $m \rightarrow \infty$.  The next proposition goes one step further.
\begin{prop}
  \label{lemm:function_z}
  We have the following: $z_m \sim m$ when $m \rightarrow \infty$.
 \end{prop}

\begin{proof}
  Let $\hat{z}_m$ be the smallest integer of the form $p^aq^b$ greater than or
  equal to $m$. Thus we have $z_m \leq m \leq \hat{z}_m$. By a theorem of
  Tijdeman~\cite{Tijdeman74a} we know that, for $m$ large enough, there exists a
  constant $C>0$ such that
  $$ \hat{z}_m - z_m < \frac{z_m}{(\log z_m)^C}, $$
  so that $0 \leq m - z_m < z_m / (\log z_m)^C $.
\end{proof}

From a greedy point of view, one might think that choosing $z_m$ for the largest
part of the \pqs{} formed as in Lemma~\ref{lemm:characterize} yields a \pqs{} of
weight $G(m)$; in which case our asymptotics problem would be solved using the
above proposition.  Unfortunately, this is not true.  In Table~\ref{tab:750}, we
see for instance that $z_{750}=729$, obtained for $(a,b) = (0,6)$, does not give
a \pqs{} of maximal weight.  Instead, the maximal weight $G(750) = 1255$ is
obtained for $(a,b) = (3,4)$.  Hence, even if $m$ is of the form $p^aq^b$, the
first part of a \pqs{} of maximal weight may be different from $m$.  For
instance $G(729) = 1255$ comes from a unique \pqs{} whose first part is $648$.
In the next Section, we study the subset $Y_m$ of $Z_m$, which yields the
\pqs{}s of maximal weight $G(m)$.

\section{On the set $Y_m$}\label{sec:Ym}
\label{sec:Ym}

According to~\eqref{def:G}, $G(m)$ is equal to $h(a,b)$ for some values $a,b$.
First, notice that these values are not necessarily unique with respect to this
property, because $h$ is not necessarily one-to-one.  For example, with $(p,q) =
(2,3)$, observe from Table~\ref{tab:750} that $h(1,5)=h(4,3)=850$.  Similarly,
with $(p,q)=(2,5)$ one has $h(0,2)=h(4,0)=31$.

Let $Y_m$ be the set of all elements $p^aq^b$ in $E\cap[0,m]$ such that
$h(a,b)=G(m)$.  As already noticed, $Y_m$ is a subset of $Z_m$.  Next recall
that $z_m=\max Z_m$ needs not be in $Y_m$.  A particular relation between $Y_m$
and $z_m$ does however exist as proved in the next proposition.

\begin{prop}
  \label{prop:droite}
  For $m \in \n^*$, let $z_m=p^{a}q^{b}$. 
  Then
  \begin{equation}
    \label{eq:Ym}
    Y_m \subset \{ p^iq^j \in Z_m : j\leq b \}
  \end{equation}
\end{prop}

\begin{proof}
Using \eqref{def:h}, we have  
\begin{align*}
  \frac {p-1}p\ h(a,b) 
     &= \frac{p-1}{p}\ \frac{q^b-1}{q-1} + \frac{q^b}{p}(p^{a+1}-1) \\[1ex]
     &= p^aq^b - q^b\ \frac{q-p}{pq-p} - \frac{p-1}{p(q-1)},  
\end{align*}
so that 
\begin{equation}
  \label{h(a,b)}
  h(a,b) = \frac{p}{p-1}\left(p^aq^b - rq^b \right) -
  \frac{1}{q-1}, \quad\text{where } r = \frac{q-p}{pq-p} \in (0,1/p)
\end{equation}
Note that $0 < r < 1/p$, because $pq-p > p(q-p)$.
As a consequence, we have
\begin{equation}
  \label{h-h}
  h(a,b)>h(a',b') \quad\Longleftrightarrow \quad
  p^aq^b - p^{a'}q^{b'}>r(q^b-q^{b'})
\end{equation} 
If $p^{a'}q^{b'} \in Z_m$ we have $z_m = p^aq^b > p^{a'}q^{b'}$.
Hence $b' > b$ implies $h(a,b)>h(a',b')$, which concludes the proof.
\end{proof}

Geometrically, Proposition~\ref{prop:droite} tells us that the points $(a,b) \in
\n^2$ such that $h(a,b) = G(m)$ cannot be located ``above'' or equivalently
``left'' of $z_m$.  In particular, when $z_m = p^a$, we have $G(m) = h(a,0) =
(p^{a+1}-1)/(p-1)$.

In Proposition~\ref{prop:Ym=2}, we will see that the set $Y_m$ has at most two
elements.  For now, let us first focus on those elements of $E$ that provide the
heaviest \pqs{} in a unique way, i.e. those for which $Y_{p^aq^b}=\{p^aq^b\}$.
The following theorem shows that the corresponding points in $\n^2$ form an
infinite area whose boundary is a particular sequence as illustrated in
Figure~\ref{fig:Z_750}.

\begin{theo}
  \label{th:ell}
  There exists a sequence $\ell=(\ell_b)_{b\in\n}$ in $\n$ such that 
  \begin{equation}
    \label{strong}
    Y_{p^aq^b} = \{p^aq^b\} \Longleftrightarrow a \geq \ell_b
  \end{equation}
  Moreover, the sequence $\ell$ is non-decreasing, unbounded, and satisfies $\ell_0=0$.
\end{theo}

\begin{proof}
  Let us first establish the following statements:
  \begin{enumerate}[(i)]
  \item For all $b\geq 0$, there exists $a\geq 0$ such that $p^aq^b\in Y_{p^aq^b}$.\\[-2ex]
  \item If $p^aq^b\in Y_{p^aq^b}$ then, for all $k\geq 1$, $Y_{p^{a+k}q^b}=\{p^{a+k}q^b\}$.
  \end{enumerate}
  
  Let $b\in\n$. As already seen, the mapping $(i,j) \mapsto p^iq^j$ is
  one-to-one. Therefore, we have $q^b-p^iq^j\geq 1$ for all $p^iq^j\in Z_{q^b}
  \setminus \{q^b\}$.  Choose $a$ such that
  $$
  p^a \geq r (q^{b}-1), \quad\text{where } r = \frac{q-p}{pq-p} \text{ as
    in~\eqref{h(a,b)}}
  $$
  Then, for all $p^iq^j\in Z_{q^b}$ with $j<b$, we have
  \begin{equation}
    \label{pa}
     p^{a}q^b - p^{a+i}q^j \geq p^a \geq r (q^b - 1) \geq r (q^b - q^j)
  \end{equation} 
  Using~\eqref{h-h}, it follows that $h(a,b) \geq h(a+i,j)$; in other words
  $p^aq^b \in p^a Z_{q^b}$.  According to Proposition~\ref{prop:droite}, we also
  have $Y_{p^aq^b}\subset\{p^iq^j\in Z_{p^aq^b}:j\leq b\}$.  Using
  Lemma~\ref{lemm:Z_m}, it is immediate to check that the latter set is
  identical to $p^aZ_{q^b}$, thus $p^aq^b\in Y_{p^aq^b}$ and (i) is proved.
  Now, if $p^aq^b\in Y_{p^aq^b}$ then replacing $a$ by $a+k$ for any $k\geq 1$
  turns~\eqref{pa} into a strict inequality.  Therefore,
  $Y_{p^{a+k}q^b}=\{p^{a+k}q^b\}$, and (ii) is proved too.  Accordingly, letting
  \begin{equation}
    \label{eq:lb}
    \ell_b = \min \left\{ a \in \n : Y_{p^aq^b} =\{p^aq^b\} \right\}  
  \end{equation}
  provides the claimed sequence $\ell$.  Since $Y_1=\{1\}$ and $1 = p^0q^0$, we
  get $\ell_0=0$.  
  \medskip

  Let us now prove that $\ell$ is non-decreasing.  Given $b\in\n$, either
  $\ell_b = 0$ thus $\ell_{b+1} \geq \ell_b$, or $\ell_b\geq 1$.  In the latter,
  there exists $p^iq^j \in Z_{p^{\ell_b-1}q^b}$ such that $j \neq b$ and
  $h(i,j)\geq h(\ell_b-1,b)$.  From~\eqref{def:h}, it is not difficult to see
  that $h(i,j+1) = qh(i,j) + 1$, and thus $h(i,j+1) - h(\ell_{b}-1,b+1) = q
  \left( h(i,j) - h(\ell_b-1,b) \right)$. Therefore $h(i,j+1) \geq
  h(\ell_b-1,b+1)$, so that $\ell_{b+1}>\ell_b-1$.  \medskip
  
  Finally suppose that $\ell$ is bounded.  This would imply that there exists an
  integer $a$ such that, for all $b\in\n$, $Y_{p^aq^b}=\{p^aq^b\}$.  The
  following statement shows that this is impossible.
  \begin{enumerate}
  \item[(iii)] For all $a\in\n$ there exists $b\in\n$ such that $h(a+\lfloor
    b\rho\rfloor,0)>h(a,b)$.
  \end{enumerate}
  \medskip

  Indeed, by Lemma~\ref{lemm:Z_m} we know that $p^{a+\lfloor b\rho\rfloor}\in
  Z_{p^aq^b}$.  Fix $a\in\n$, choose $b\in\n^*$, and set $a'=a+\lfloor
  b\rho\rfloor$.  According to~\eqref{h-h}, and denoting by
  $\{b\rho\}=b\rho-\lfloor b\rho\rfloor$ the fractional part of $b\rho$, we have
  \begin{align}
    h(a,b) < h(a',0) 
       & \Leftrightarrow p^aq^b-p^{a'} < r(q^b-1) \notag \\
       & \Leftrightarrow p^{\lfloor b\rho\rfloor} > q^b \left(1 -
         r/p^{a} \left(1- 1/q^b \right)\right) \notag \\
       & \Leftrightarrow  \{ b\rho \} <
       -\log_p\left(1-r/p^{a}\left(1-1/q^b\right)\right)
       \label{majfrac}
  \end{align}
  Now observe that sequence $(\phi_{a,b})_{b \in \n}$ defined by 
  \begin{equation}
    \label{eq:phiab}
    \phi_{a,b} = -\log_p \left(1-\frac r{p^{a}}\left(1-\frac1{q^b}\right)\right),
  \end{equation}
  where $r = (q-p)/(pq-p)$, is increasing, with $\phi_{a,0}=0$.  Since $\rho$ is
  irrational, we know that $(b\rho)_{b\in\n}$ is equidistributed modulo 1. Thus,
  there exists $b>0$ such that $\{b\rho\} < \phi_{a,1}$, hence $\{b\rho\} <
  \phi_{a,b}$, which using~\eqref{majfrac} concludes the proof.
\end{proof}

Let us anticipate a result of the next section, implying that the sequence
$\ell$ is completely known as soon as we can compute its \emph{jump indices},
that is, the values $b>0$ for which $\ell_b>\ell_{b-1}$.  Indeed, we shall
establish with statement~\eqref{lgap} that if $b$ is a jump index of $\ell$ then
$\ell_b=\left\lfloor\alpha(b) \right\rfloor$, where
\begin{equation}
  \label{alpha}
  \alpha(b)= \log_p \frac{q^{b}-1}{q^b-p^{\lfloor b\rho\rfloor}}+ \log_p \frac{q-p}{q-1}
\end{equation}

Computing the jump indices of $\ell$ may be done recursively as shown in the
next corollary.  Referring to the sequence $\phi$ defined in~\eqref{eq:phiab},
let the mapping $\beta$ be defined on $\n$ by
\begin{equation}
  \label{beta}
  \beta(a) = \min\left\{ j\in\n^{*}:\;\{j\rho\} < \phi_{a,j} \right\}
  \end{equation}
\begin{coro}
  \label{coro:gaps}
  The increasing sequence $(j_k)_{k\in\n}$ of the jump indices of $\ell$
  satisfies
  $$
  j_0=\beta(0), \quad j_{k+1}=\beta(\ell_{j_k})
  $$
\end{coro}
\begin{proof}
  Given any $a\in\n$, we know that $Y_{p^a}=\{p^a\}$ by
  Proposition~\ref{prop:droite}.  Moreover, letting $b$ be incremented by 1 from
  0 iteratively, it follows from \eqref{eq:Z_qm} and \eqref{h-h} that
  $Y_{p^aq^b}=\{p^aq^b\}$ as long as $h(a,b)>h(a+\lfloor b\rho\rfloor,0)$.  As
  it is shown in part (iii) of the proof of Theorem~\ref{th:ell}, the latter
  inequality is equivalent to $\{b\rho\}<\phi_{a,b}$.  Accordingly,
  $Y_{p^aq^b}=\{p^aq^b\}$ if $b<\beta(a)$.  Hence $\ell_{\beta(a)-1}\leq
  a<\ell_{\beta(a)}$, so that $\beta(a)$ is a jump index for $\ell$.  The result
  follows immediately since $\ell$ is non-decreasing.
\end{proof}

As claimed before, we next show that $Y_m$ has at most 2 elements.  This might
be established by directly using~\eqref{h-h} and studying the diophantine
equation
$$
p^aq^b - p^{c} = r(q^b-1), \quad\text{where } 
r = \frac{q-p}{pq-p}
$$
Unfortunately, the latter is seemingly not easy to cope with, whereas
Theorem~\ref{th:ell} proves handy.

\begin{prop}
  \label{prop:Ym=2}
  For all $m\in\n$, the set $Y_m$ has either one or two elements. 
\end{prop}

\begin{proof}
  Assume $\#Y_m \geq 2$ and denote by $p^aq^b$ its greatest element.  Since
  $Y_m=Y_{p^aq^b}$, we have $a<\ell_b$ by definition of sequence $\ell$ in
  Theorem~\ref{th:ell}.  To be more precise, statement (ii) in the proof of this
  theorem even tells us that $a=\ell_b-1$.  Now let $p^{c}q^{d}$ be the second
  greatest element in $Y_m$.  According to Proposition~\ref{prop:droite} we must
  have $d<b$, thus $c>a$ by Lemma~\ref{lemm:Z_m}.  Since $\ell$ is
  non-decreasing, it follows that $\ell_{d}\leq\ell_b$.  Since $a=\ell_b-1$ we
  get $c\geq\ell_{d}$, which means that $Y_{p^{c}q^{d}}=\{p^{c}q^{d}\}$ from
  Theorem~\ref{th:ell}.  Therefore, there cannot exist a third element in $Y_m$
  as it would also be in $Y_{p^{c}q^{d}}$.
\end{proof}

\section{Asymptotic behavior of $G$}\label{sec:asymptotics}

In this section, our goal is to prove that $G(m)$ is equivalent to $mp/(p-1)$ as
$m$ tends to infinity, independently of $q$.  As a simple first step, let us
exhibit a sharp upper bound for $G$.

\begin{lemm}
  \label{lemm:upperbound}
  For all $m \in \n$ and all $n \in Y_m$, we have $ G(m)< np/(p-1)$. In
  particular,
  $$ \limsup_{m\rightarrow\infty} \frac{G(m)}{m} = \frac{p}{p-1} $$ 
\end{lemm}

\begin{proof}
  Let $n=p^aq^b\in Y_m$. According to~\eqref{h(a,b)} we have
  \begin{equation}
    \label{h(n)-pn}
    h(a,b) - \frac{np}{p-1} = - \left(\frac{rpq^b}{p-1} + \frac{1}{q-1}
    \right) < 0, \quad \text{where } r \in (0,1/p)
  \end{equation} 
  Hence $G(m)=h(a,b) < np/(p-1)$. Since $n\leq m$, it follows that $G(m)/m \leq
  p/(p-1)$.  To conclude, observe that for $m=p^a$ we have $G(p^a) =
  (p^{a+1}-1)/(p-1)$ from Prop~\ref{prop:droite}.  Therefore,
  $\lim_{a\rightarrow\infty}G(p^a)/p^a= p/(p-1)$.
\end{proof}

Let us now define a mapping $y$ as follows: For all $m$, let $y_m$ denote the
smallest integer of the form $p^aq^b$ such that $G(m)=h(a,b)$, that is
\begin{equation}
  \label{ym}
  y_m=\min Y_m
\end{equation}

According to Proposition~\ref{prop:droite}, $y_m$ is also the element of $Y_m$
with the smallest exponent in $q$.  We shall next give a characterization of
$y_m$ using the sequence $\ell$ defined in Theorem~\ref{th:ell}.  Recall that
this sequence is defined by $\ell_b=\min\{a\in\n:Y_{p^aq^b}=\{p^aq^b\}\}$ and
satisfies \eqref{strong}.  Since $\ell$ is non-decreasing, the sequence
$(p^{\ell_b}q^b)_{b\in\n}$ is increasing.  We may thus define, for all $m\in\n$,
\begin{equation}
  \label{mell}
  m_\ell = \max\{b\in\n:p^{\ell_b}q^b\leq m\}
\end{equation}

\begin{theo}
  \label{prop:y_m}
  For all $m \in \n$, we have
  \begin{equation}
    \label{eq:ym}
    y_m = \max \{ p^aq^b \in Z_m : b \leq m_\ell \}
  \end{equation}
  Moreover, let $\bar{a}=\lfloor\log_p m-m_\ell\rho\rfloor$ and $\bar{m}=\lfloor
  m/p^{\bar{a}}\rfloor$. Then $y_m=p^{\bar{a}}z_{\bar{m}}$.
\end{theo}

\begin{proof}
  Let $y_m=p^iq^j$.  Since $y_m=\min Y_m$, we have $Y_{y_m}=\{y_m\}$ thus $i\geq
  \ell_j$ using~\eqref{strong}.  Suppose $j > m_\ell$, then $p^iq^j \geq
  p^{\ell_j}q^j$, and thus $p^iq^j >m$ from~\eqref{mell}, which contradicts the
  fact that $y_m \leq m$. Therefore, $j \leq m_\ell$.

  Now consider any $p^aq^b\in Z_m$ such that $b\leq m_\ell$.  Since $\ell$ is
  non-decreasing we have $\ell_{m_\ell}\geq\ell_b$.  By Lemma~\ref{lemm:Z_m},
  there exists $k\in \n$ such that $p^kq^{m_\ell}\in Z_m$, and we have $k \geq
  \ell_{m_\ell}$ because $p^{\ell_{m_\ell}}q^{m_\ell} \leq m$.  Since $p^aq^b\in
  Z_m$, condition $b \leq m_\ell$ implies that $a \geq k$ by
  Lemma~\ref{lemm:Z_m} again, so that $a \geq \ell_{m_\ell}\geq \ell_b$.
  Therefore, $Y_{p^aq^b}=\{p^aq^b\}$.  Supposing $p^aq^b>p^iq^j$ would then
  imply that $h(i,j)<h(a,b)$, contradicting the definition of $y_m$.  Thus
  $p^aq^b\leq p^iq^j$, and \eqref{eq:ym} is established.
  
  Accordingly, $j=\lfloor\log_p m-i\rho\rfloor\leq\bar{a}$, so that
  $y_m/p^{\bar{a}}$ is an element of $E$.  Therefore, $y_m/p^{\bar{a}}\leq
  m/p^{\bar{a}}$ implies $y_m/p^{\bar{a}}\leq\lfloor m/p^{\bar{a}}\rfloor
  =\bar{m}$, which in turn implies $\lfloor m/p^{\bar{a}}\rfloor\leq
  z_{\bar{m}}$.  We thus get
  $$
  y_m\leq p^{\bar{a}}z_{\bar{m}}\leq p^{\bar{a}}\bar{m}\leq m
  $$
  Let $z_{\bar{m}}=p^aq^b$.  To conclude the proof by using \eqref{eq:ym} again,
  it suffices to show that $b\leq m_\ell$.  Since $p^{\bar{a}}q^{m_\ell}\in
  Z_m$, we have $p^{\bar{a}}q^{m_\ell}\leq m<p^{\bar{a}}q^{m_\ell+1}$, so that
  $q^{m_\ell}\leq \bar{m}<q^{m_\ell+1}$.  Thus
  $\lfloor\log_q\bar{m}\rfloor=m_\ell$, hence $b\leq m_\ell$ by
  Lemma~\ref{lemm:Z_m}.
\end{proof}

Comparing with Proposition~\ref{prop:droite}, characterization \eqref{eq:ym} of
$y_m$ no more depends on $z_m$.  Moreover, it provides a first improvement of
Lemma~\ref{lemm:upperbound}.
\begin{coro}  
  For all $m \in \n$, we have
  \begin{equation}
    \label{eq:Gm:ym}
    \frac{G(m)}{y_m} \sim \frac{p}{p-1} \quad\text {as } m \rightarrow \infty
  \end{equation}
\end{coro}

\begin{proof}
  For $m\in\n$, let $y_m=p^{a_m}q^{b_m}$.  According to~\eqref{h(n)-pn} we have
  $$
  \frac{G(m)}{y_m} - \frac{p}{p-1}
  = -\frac{t_m}{p^{a_m}}, \quad\text{where }
  t_m = \frac{q-p}{(p-1)(q-1)}+\frac{1}{q^{b_m}(q-1)}
  $$
  Observe that $t_m \in \left(\frac{q-p}{(p-1)(q-1)},\frac{1}{p-1}\right]$
  is uniformly bounded. To conclude the proof, we need to show that $a_m
  \rightarrow \infty$ as $m \rightarrow \infty$.  According to
  Theorem~\ref{prop:y_m}, we have $b_m \leq m_\ell$, thus
  $a_m\geq\ell_{m_\ell}$.  Since $m_\ell$ goes to infinity with $m$, and since
  $\ell$ is non-decreasing and unbounded by Theorem~\ref{th:ell}, $a_m$ goes to
  infinity with $m$ too. Hence the claim.
\end{proof}

According to the latter result and Proposition~\ref{lemm:function_z}, the final
task consists in showing that $y_m\sim z_m$.  This is done next, so that our
main claim is established.

\begin{theo}
  \label{theo:Gm/m}
  For all $m \in \n$, we have
  $$ \frac{G(m)}{m} \sim \frac{p}{p-1} \quad\text{as } m \rightarrow
  \infty $$
\end{theo}

\begin{proof}
  Assume $y_m \neq z_m$.  According to Theorem~\ref{prop:y_m}, the elements in
  $Z_m$ that exceed $y_m$ are of the form $p^iq^j$ with
  $m_\ell<j\leq\lfloor\log_q m\rfloor$.  Let us sort these elements together
  with $y_m$ in an increasing sequence $(n_0=y_m, n_1, \dots, n_N=z_m)$, so that
  $n_0=y_m$ and $n_N=z_m$.  Observe that the elements of this sequence are
  consecutive elements of $E$ for the usual order.  As soon as $m$ is large
  enough, we know by Tijdeman's result already mentioned~\cite{Tijdeman74a} that
  $n_{i+1} - n_i \leq n_i/(\log n_i)^C$ for an explicitly computable constant
  $C>0$.  Therefore,
  $$
  z_m - y_m 
  \leq \sum_{i=0}^{N-1}\frac{n_i}{(\log n_i)^C}
  \leq \sum_{i=0}^{N-1}\frac{py_m}{(\log \frac{m}{p})^C}
  $$
  Accordingly, for all $m \in \n$,
  \begin{equation}
    \label{ecart}
    0 \leq z_m - y_m \leq (\lfloor\log_q m\rfloor - m_\ell)\, \frac{py_m}{(\log \frac{m}{p})^C}
  \end{equation}
  
  At this point, what remains to be proved is that $\lfloor\log_q m\rfloor -
  m_\ell$ grows asymptotically slower than $(\log \frac{m}{p})^C$ so that $z_m
  \sim y_m$, and to conclude using~\eqref{eq:Gm:ym} and
  Lemma~\ref{lemm:function_z}.  Recall that $m_\ell$ is defined as the largest
  value $b$ such that $p^{\ell_b}q^b \leq m$. Therefore, we have
  $$
  p^{\ell_{m_\ell}}q^{m_\ell} \leq m < p^{\ell_{m_{\ell+1}}}q^{m_\ell+1}
  $$
  Equivalently, using $\rho = \log{q}/\log{p}$, we have
  \begin{equation}
    \label{C0}
    \frac{\ell_{m_\ell}}{\rho}+m_\ell \leq \log_q m <\frac{\ell_{m_\ell+1}}{\rho}+m_\ell+1
  \end{equation}
  so that 
  \begin{equation}
    \label{C}
    \lfloor\log_q m\rfloor-m_\ell<\frac{\ell_{m_\ell+1}}{\rho}+1
  \end{equation}
  
  It thus remains to evaluate the terms in sequence $\ell$.  For that purpose,
  we first give an explicit formula for $\ell$, valid at the jumps of $\ell$. We
  claim that for all $b \in \n^*$,
  \begin{equation}
    \label{lgap}
    \ell_b > \ell_{b-1} \quad\Rightarrow\quad
    \ell_b = \left\lfloor \log_p \frac{(q-p)(q^b-1)}{(q-1)(q^b-p^{\lfloor
          b\rho\rfloor})}  \right\rfloor
  \end{equation}

  Indeed, assume that $\ell_b > \ell_{b-1}$. 
  Then, for all $p^iq^j \in Z_{p^{\ell_{b-1}}q^{b-1}}$, we know from~\eqref{h-h}
  that
  $$
  p^{\ell_{b-1}}q^{b-1}-p^iq^j>r(q^{b-1}-q^j)
  $$
  Multiplying both sides by $q$ yields, for all $p^iq^j\in
  qZ_{p^{\ell_{b-1}}q^{b-1}}$
  $$
  p^{\ell_{b-1}}q^{b}-p^iq^j>r(q^{b}-q^j)
  $$
  Now, $\ell_b>\ell_{b-1}$ implies that there exists an element $p^iq^j \in
  Z_{p^{\ell_{b-1}}q^{b}}$ for which the latter inequation does not hold.  By
  Lemma~\ref{lemm:Z_m} and the definition of $Z_{qm}$ in~\eqref{eq:Z_qm}, this
  element must be $p^{\ell_{b-1} + \lfloor b\rho \rfloor}$, so that
  $$
  p^{\ell_{b-1}}(q^{b}-p^{\lfloor b\rho\rfloor})\leq r(q^{b}-1)
  $$
  In fact, note that this inequality does not only hold for $p^{\ell_{b-1}}$ ;
  by definition of $\ell$, it remains valid for $p^{\ell_{b-1}+1}, \dots,
  p^{\ell_b-1}$.  Accordingly, we get
  $$
  p^{\ell_{b}-1}(q^{b}-p^{\lfloor b\rho\rfloor})\leq
  r(q^{b}-1)<p^{\ell_{b}}(q^{b}-p^{\lfloor b\rho\rfloor})
  $$
  which proves claim~\eqref{lgap}.  
  \medskip

  It follows from~\eqref{lgap} that, for any $b$ such that $\ell_b>\ell_{b-1}$,  
  \begin{equation}
    \label{A}
    \ell_b < \log_p\frac{q^b}{q^b-p^{\lfloor b\rho\rfloor}}
  \end{equation} 
  Using another result of Tijdeman (see~\cite[Theorem~1]{Tijdeman73}), we know
  that, as soon as $p^{\lfloor b\rho\rfloor} > 3$, there exists another explicit
  constant $C' > 1$ such that
  \begin{equation}
    \label{B}
    q^b-p^{\lfloor b\rho \rfloor} \geq 
    \frac{p^{\lfloor b\rho \rfloor}}{\left(\log p^{\lfloor b\rho\rfloor}\right)^{C'}}
  \end{equation} 
  Therefore, since $q^b=p^{b\rho}$,~\eqref{A} and~\eqref{B} imply that
  \begin{equation}
    \label{lbound}
    \ell_b < \log_p \frac{q^b(\log p^{\lfloor b\rho\rfloor})^{C'}}{p^{\lfloor
        b\rho\rfloor}}
    = \{b\rho\} + \frac{C'}{\log p} \log p^{\lfloor b\rho\rfloor}
    < 1 + \frac{C'}{\log p}\log\log q^b
  \end{equation} 
  
  Putting all this together, we get the claimed result. 
  Indeed, let $b$ be the smallest index such that $\ell_{m_\ell+1} = \ell_b$. 
  Since $\ell_b > \ell_{b-1}$, we have
  \begin{equation}
    \label{algo}
    \ell_{m_\ell+1} = \ell_b
    <  1+\frac{C'}{\log p}\log\log q^{b}
    \leq 1+\frac{C'}{\log p}\log\log q^{m_\ell+1}
    \leq 1+\frac{C'}{\log p}\log\log qm
  \end{equation} 
  Using~\eqref{C} we get 
  \begin{equation}\label{bm-ml}
    \lfloor \log_q m \rfloor - m_\ell < 1 + \frac{1}{\rho} +\frac{C'}{\log
      q}\log\log qm =o\left((\log m/p)^{C}\right)
  \end{equation} 
  which implies, using~\eqref{ecart}, that $y_m \sim z_m$ and concludes the proof.
\end{proof}
Note that the above proof mainly relies on the fact that the sequence $\ell$ is
non-decreasing and that, due to the lower bound in \eqref{B} essentially, it
grows very slowly.  The theorem of Tijdeman that provides this lower bound
hinges on a result of Fel'dman about linear forms in logarithms.  More recent
results of Laurent et alii~\cite{LauMigNes95} about such forms in two logarithms
allow one to make precise the value of the effective constant $C'$ in \eqref{B}.
Nevertheless, this value remains large and does not seem convenient in order to
compute $m_\ell$ using~\eqref{bm-ml}, in particular when $m$ is not very large.
The algorithm presented in the next section provides one with an alternative
method.

\section{Computing $y_m$ and $G(m)$}
\label{sec:computing}

Using the mapping $h$, computing $G(m)$ is straightforward as soon as an element
of $Y_m$ is known, in particular $y_m$.  In Theorem~\ref{prop:y_m}, we proved
that $y_m = p^{\bar{a}} z_{\bar{m}}$, where $\bar{a}$ and $z_{\bar{m}}$ both
depend on $m_\ell$. Once $m_\ell$ is known, computing $z_{\bar{m}}$ (the
greatest element in $Z_{\bar{m}}$) can be done efficiently with an algorithm
explained in~\cite{BerImb09:dmtcs}. We shall establish a slightly different and
  simpler version of that algorithm at the end of this section. 

Computing $m_\ell$ requires to compute the values of $\ell$ (see~\eqref{mell}).
Theorem~\ref{th:beta} given below asserts that the jump indices of $\ell$ are
denominators of convergents of $\rho$.  Furthermore, the relation $\ell_b =
\lfloor \alpha(b) \rfloor$, see~\eqref{alpha}, also holds for \emph{all}
denominators of both even primary convergents of $\rho$ and their mediants.
This provides an explicit method for computing any term of $\ell$, stated in
Corollary~\ref{coro:ellb}.

Some known facts about the convergents of $\rho$ are thus needed, let us recall
them (see, e.g.,~\cite{Lang95} or~\cite{AllSha03} for more details).  Let
$[a_0,a_1,...]$ be the regular continued fraction expansion of $\rho$, and for $i
\geq 0$,  denote by $h_i/k_i$ the $i^\mathrm{th}$ principal convergent of
$\rho$.
It is well known that the sequence $(h_i/k_i)_{i\geq0}$ converges to $\rho$ and
satisfies
$$|k_i\rho - h_i| = (-1)^i(k_i\rho - h_i),$$
and
$$
\frac{1}{k_{i}+k_{i+1}} < \left| k_i\rho - h_i \right| < \frac{1}{k_{i+1}}
$$
Given $i\geq 0$, the intermediate convergents of $h_i/k_i$, sometimes
  referred to as mediants, are the rational numbers ${h_{i,j}}/{k_{i,j}}$, given
  by
\begin{equation}
  \label{def2cv}
  h_{i,j} = h_i + jh_{i+1}, \quad k_{i,j} = k_i + jk_{i+1}, \quad\text{for } 0 <
  j < a_{i+2} 
\end{equation}
Let $(h_{2i}/k_{2i})_{i\geq 0}$ denote the sequence of convergents of $\rho$
  of even index.  We define $(H_n/K_n)_{n\in\n}$ as the increasing sequence of
  all convergents of $\rho$ of even index together with their intermediate
  convergents.
   It is also known (see~\cite[Theorem~2]{Richards81}) that
$(H_n/K_n)_{n \in \n}$ is the best approximating sequence of $\rho$ from below,
that is, its terms are characterized by the following property: For each
$n\in\n$ and integers $h$ and $k$, we have
\begin{equation}
  \label{convergents}
  \frac{H_n}{K_n} < \frac{h}{k} < \rho \quad \Longrightarrow \quad k > K_n
\end{equation}

We shall need two immediate consequences of~\eqref{convergents}.  The first one
is that, while the sequence $(K_n)_{n\in\n}$ increases to infinity, the sequence
$(\{K_n\rho\})_{n\in\n}$ decreases to 0.  Indeed, property~\eqref{convergents}
implies that $H_n = \lfloor K_n\rho \rfloor$, so that~\eqref{def2cv} implies,
for $0 \leq j < a_{i+2}$,
\begin{equation}
  \label{cvrho}
  \{k_{2i,j+1}\rho\}-\{k_{2i,j}\rho\} = k_{2i+1}\rho - h_{2i+1} \in (-1/k_{2i+2},0)
\end{equation}
The second one rephrases the sufficient condition in~\eqref{convergents}: If
$\{b\rho\}\leq\{j\rho\}$ holds for all integers $0 < j \leq b$, then
$b$ is a term of $(K_n)_{n\in\n}$.  Indeed, let $h,k$ be such that
$$
\frac {\lfloor b\rho\rfloor}b<\frac hk<\rho
$$
Then, since $h<k\rho$, the above inequalities still hold for $h=\lfloor
k\rho\rfloor$. This implies $\{k\rho\}<\frac{k}{b}\{b\rho\}$.  Supposing $k\leq
b$ yields $\{k\rho\}<\{b\rho\}$, which contradicts our hypothesis. Hence
$\lfloor b\rho\rfloor/b$ satisfies~\eqref{convergents}.

We can now establish our main claims.  According to~\eqref{lgap}, the values of
$\ell$ are known explicitly at every jump index $j$ of $\ell$. In these cases,
we have $\ell_j = \left\lfloor \alpha(j) \right\rfloor$ where, as already
defined in \eqref{alpha}:
$$
\alpha(b)= \log_p \frac{q^{b}-1}{q^b-p^{\lfloor b\rho\rfloor}}+ \log_p
\frac{q-p}{q-1}
$$
 
\begin{theo}
  \label{th:beta}
  Every jump index of $\ell$ is a term of the sequence $(K_n)_{n\in\n}$.
  Moreover, for each $K\in(K_n)_{n\in\n}$, we have
  $\ell_K = \left\lfloor \alpha(K) \right\rfloor$.
\end{theo}

\begin{proof}
  According to Corollary~\ref{coro:gaps}, the jump indices $j_k$ of $\ell$ can
  be computed starting from $j_0 = \beta(0)$ and iterating $j_{k+1} =
  \beta(\ell_{j_k})$, where $\beta(a) = \min \{j \in \n^* : \{j\rho\} <
  \phi_{a,j} \}$ (see~\eqref{eq:phiab} and~\eqref{beta}).  Let us fix $a\in\n$.
  Since the sequence $(\phi_{a,i})_{i\geq0}$ increases from 0 and the sequence
  $(\{K_i\rho\})_{i\geq0}$ decreases to 0, there exists a unique $n$ such that
  $\{K_n\rho\} < \phi_{a,K_n}$ and $\{K_i\rho\}\geq\phi_{a,K_i}$ for all $i<n$,
  if any.  In particular, $K_n\geq \beta(a)$.  We next establish that
  $K_n=\beta(a)$.  This is clear if $n=0$ since $K_0=1$ and $\beta(a)\geq 1$. In
  the following, we assume $n\geq 1$.

  Let $b=\beta(a)$ for short, and suppose $b<K_n$.  Let $K=K_{n-1}$ for short
  again, and let $c=\min\{j>0:\{j\rho\}<\{K\rho\}\}$.  For all $i<c$ we have
  $\{i\rho\}\geq\{K\rho\}>\{c\rho\}$, so that $c\in(K_i)$, which implies $c=K_n$
  since $(\{K_i\rho\})$ decreases.  Therefore, $b<K_n$ forces
  $\{b\rho\}\geq\{K\rho\}$, so that
  \begin{equation}
    \label{ecarts}
    \phi_{a,b}>\{b\rho\}\geq\{K\rho\}\geq\phi_{a,K}.
  \end{equation} 
  Since $\phi_{a,i}$ increases with $i$, we get $K<b<K_n$.  Property
  \eqref{convergents} implies that $\lfloor b\rho\rfloor/b< \lfloor
  K_n\rho\rfloor/K_n$.  Since we cannot have $\lfloor K\rho\rfloor/K<\lfloor
  b\rho\rfloor/b< \lfloor K_n\rho\rfloor/K_n$ (\cite{Richards81}, (ii) of Lemma~1), it
  follows that $\lfloor b\rho\rfloor/b< \lfloor K\rho\rfloor/K$, that is,
  $\{b\rho\}/b>\{K\rho\}/K$.  Thus
  $$
  \{b\rho\}-\{K\rho\}>\frac{b-K}{K}\{K\rho\}\geq
  \frac{\phi_{a,K}}{K}>\frac{x(1-q^{-K})}{K\log p},
  $$ 
  where $x=r/p^a$ for short. Nevertheless, 
  $$
  \phi_{a,b}-\phi_{a,K}<\phi_{a,\infty}-\phi_{a,K}=
  \log_p\bigg(1+\frac x{1-x}q^{-K}\bigg)<\frac{x}{(1-x)q^{K}\log p}.
  $$
  According to \eqref{ecarts} we should thus have
  $$
  \frac{(1-q^{-K})}{K}<\frac{1}{(1-x)q^{K}}, 
  $$
  which would imply, since $x\leq r$, 
  $$
  q-1<\frac{q^{K}-1}{K}<\frac1{1-x}\leq\frac{p(q-1)}{q(p-1)}<q-1. 
  $$
  Therefore, $K_n=\beta(a)$ as claimed, which proves the first assertion of the
  Theorem.
  
  Now, let $K\in(K_n)$.  On one hand, $\ell_K\leq\lfloor\alpha(K)\rfloor$ holds.
  Indeed, there is a unique jump index $K^*$ of $\ell$ such that
  $\ell_K=\ell_{K^*}$, and $K^*\leq K$ because $\ell$ is non-decreasing.
  According to the first assertion of the Theorem, $K^*\in(K_n)$.  Since $(K_n)$
  increases and $(\{K_n\rho\})$ decreases, $(\alpha(K_n))$ is increasing, thus
  $(\lfloor\alpha(K_n)\rfloor)$ is non-decreasing.  Therefore,
  $\ell_K=\ell_{K^*}=\lfloor\alpha(K^*)\rfloor\leq\lfloor\alpha(K)\rfloor$.  On
  the other hand, we also have $\lfloor\alpha(K)\rfloor\leq \ell_K$.  Indeed,
  letting $a=\lfloor\alpha(K)\rfloor$ for short, we have $a\leq \alpha(K)$, that
  is,
  $$
  p^a\leq \frac{(q-p)(q^K-1)}{(q-1)(q^K-p^{\lfloor K\rho\rfloor})}. 
  $$
  Recalling \eqref{h-h}, this also reads
  $$
  (p^{a-1}q^K-p^{a-1+\lfloor K\rho\rfloor})\leq r(q^K-1). 
  $$
  Since $\ell_K\geq 0$, we may assume $a\geq 1$.  Letting $a'=a+\lfloor
  K\rho\rfloor$, the above inequality means that
  $$
  h(a-1,K)\leq h(a'-1,0).
  $$
  Finally notice that $p^{a'-1}<p^{a-1}q^K$.  Therefore,
  $Y_{p^{a-1}q^K}\neq\{p^{a-1}q^K\}$, so that $a-1<\ell_K$ by \eqref{strong}.
  Thus $\lfloor\alpha(K)\rfloor=a\leq\ell_K$ as claimed.
\end{proof}

Accordingly, computing $\ell_b$ for an arbitrary $b$ only requires applying
$\alpha$ to the largest term of $(K_n)_{n\in\n}$ not exceeding $b$.  More
explicitly:

\begin{coro}
  \label{coro:ellb}
  Given $b\in\n$,  let $s$ and $t$ be the integers defined by 
  $$
  k_{2s} \leq b < k_{2s+2}, \quad\text{and}\quad t = \left\lfloor
    \frac{b-k_{2s}}{k_{2s+1}} \right\rfloor
  $$
  Then $\ell_b=\lfloor\alpha(k_{2s,t})\rfloor$. 
\end{coro}

Let us finally turn to the computation of $y_m=p^aq^b$ and $G(m)=h(a,b)$.  Of
course, writing these values requires $O(\log m)$ bits, but we show that the $a$
and $b$ exponents of $y_m$ can be obtained with $O(\log\log m)$ operations
involving numbers of $O(\log\log m)$ bits.

According to Lemma~\ref{lemm:Z_m}, for each integer $b\in[0,\log_m]$ there is
a unique integer $a$ such that $p^aq^b\in Z_m$, given by $a = \left\lfloor
  \log_p m - b\rho \right\rfloor$.  Let us define, for any $b \in \n$,
$$ \zeta(b) = p^{\left\lfloor \log_p m - b\rho \right\rfloor} q^b $$
In particular, for each $0 \leq b \leq \left\lfloor \log_q m \right\rfloor$,
$\zeta(b)$ is the element in $Z_m$ whose exponent in $q$ is $b$.  In the
  example given in Figure~\ref{fig:Z_750} for $(p,q) = (2,3)$ and $m = 750$, the
  terms of $\zeta$, given for $0 \leq b \leq \left\lfloor \log_q m \right\rfloor
  = 6$ are $(512, 384, 576, 432, 648, 486, 729)$.  Let us now define the
sequence $(b_i)_{i\in\n}$ as follows: $b_0 = 0$ and for $i \geq 0$
$$
  b_{i+1} = \min \{b\in\n: b>b_i \text{ and }
  \zeta(b)>\zeta(b_i) \}
$$
This sequence is the basis for the algorithm in~\cite{BerImb09:dmtcs} that computes
$z_m$.  The following lemma tells us that this same algorithm can also be used
to compute $y_m$.

\begin{lemm}
  \label{lemm:basis}
  Let $I = \max\{i\in\n : b_i\leq \log_q m \}$ and $J = \max\{i\in\n : b_i\leq
  m_\ell\}$, then $\zeta(b_I) = z_m$ and $\zeta(b_J) = y_m$.
\end{lemm}

\begin{proof}
  Let $z_m=\zeta(b^*)$.  Since $\zeta(b_I)\in Z_m$, we have $\zeta(b^*)\geq
  \zeta(b_{I})$.  Suppose $\zeta(b^*)> \zeta(b_{I})$, then $b^*\geq b_{I+1}$,
  which contradicts the definition of $I$.  Thus $\zeta(b^*)= \zeta(b_{I})$,
  hence $b^*= b_{I}$.  Finally, it follows from Theorem~\ref{prop:y_m} that
  $\zeta(b_J)=y_m$.
\end{proof}
In our running example (see Figure~\ref{fig:Z_750}) it can be read directly
  from Table~\ref{tab:750} that sequence $(b_i)$ starts with $(0,2,4,6)$. We
  thus have $I=3$ and, since $m_\ell=4$ (to be read on Figure~\ref{fig:Z_750}),
  $J=2$.

In order to compute successive terms of $(b_i)_{i\in\n}$, we next state a
modified version of the result in~\cite{BerImb09:dmtcs}, which provides a simple
bound on the number of steps required to get $z_m$ from $(b_i)_{i\in\n}$ without
any particular assumption on the partial quotients of $\rho$.  

For each principal convergent $h_s/k_s$ of $\rho$, let $\eps_s=|k_s\rho-h_s|$.
We known that $(\eps_s)_{s\in\n}$ is strictly decreasing and
  converges towards 0. Besides, we also know that the convergents of even
  index approach $\rho$ from below (see~\eqref{convergents}) , whereas those
  of odd index approach $\rho$ from above. Hence $\eps_{2s} = k_{2s}\rho -
  h_{2s}$, while $\eps_{2s+1} = -k_{2s+1}\rho + h_{2s+1}$.  Thus we have
  $\{k_{2s}\rho\}=\eps_{2s}$ and, for all $0 \leq t \leq a_{2s+2}$,
\begin{equation}
  \label{eq:eps}
  \{k_{2s,t}\rho\} = \eps_{2s} - t \eps_{2s+1} =
  \eps_{2s+2}+(a_{2s+2}-t)\eps_{2s+1}
\end{equation}

Our main theorem regarding the computation of $y_m$ and $G(m)$ can now be
  established.

\begin{theo}
  \label{th:seqd}
  For all $i\geq 0$, let $d_{i+1}=b_{i+1}-b_{i}$. The sequence $(d_{i})_{i\geq
    1}$ is non-decreasing ; its terms all belong to $(K_n)_{n\in\n}$ and satisfy
  the following properties:
  \begin{itemize}
  \item[(i)] For each $s\in\n$, there exists at most one value of $t$ in
    $(0,a_{2s+2})$ such that $k_{2s,t}$ belongs to $(d_{i})_{i\geq 1}$. If
    $d_{i+1}=k_{2s,t}$ with $0<t<a_{2s+2}$, then $t=t_s$ is given by
    \begin{equation}
      \label{ts}
      t_s = \left\lceil\frac{\eps_{2s}-\{\log_p m-b_{i}\rho\}}{\eps_{2s+1}}\right\rceil
    \end{equation}   
    
  \item[(ii)] If $d_{i+1} = k_{2s}$ with either $i=0$ or $d_{i+1}>d_{i}$, then
    $k_{2s}$ occurs $n_s$ consecutive times in $(d_{i})_{i\geq 1}$, where
    \begin{equation}
      \label{ns}
      n_s = \left\lfloor \frac {\{\log_p m-b_{i}\rho\}}{\eps_{2s}} \right\rfloor
    \end{equation}
    Moreover, if either $s=0$ or $d_{i}=k_{2s-2,t}$ with $0<t<a_{2s}$, then
    $n_s\leq a_{2s+1}$, else $n_s\leq 1+a_{2s+1}$.
  \end{itemize}
\end{theo}

\begin{proof} 
  Let $r_i=\{\log_p m-b_{i}\rho\}$ for short.  For all integers $b$ such that
  $b_i < b \leq \log_q m$, we have \footnote{Note that $0\leq y\leq x$ implies
    $\lfloor x-y\rfloor +y -\lfloor x\rfloor = \lfloor \{x\}-\{y\}\rfloor+\{y\}$, 
    and expand~\eqref{diff-zeta} as $\log_p m - b\rho = \log_p m -
      (b-b_i)\rho - b_i\rho$.}
  \begin{align}
    \log_p\zeta(b)-\log_p\zeta(b_i) 
    = & \lfloor\log_p m-b\rho\rfloor +b\rho - \lfloor\log_p m-b_i\rho\rfloor
    -b_i\rho \label{diff-zeta}\\
    = & \lfloor r_i-\{(b-b_i)\rho\}\rfloor+\{(b-b_i)\rho\}\notag
  \end{align}
  Thus $\zeta(b)>\zeta(b_i)$ if, and only if, $\{(b-b_i)\rho\}\leq r_i$.  For
    $b = b_{i+1}$, this reads $\zeta(b_{i+1}) > \zeta(b_i) \iff
    \{d_{i+1}\rho\} \leq r_i$. Therefore, any positive integer $c < d_{i+1}$
    satisfies $\{c\rho\}>r_i$, hence 
  \begin{equation}
    \label{di}
    d_{i+1} = \min \{ d\in\n^* : \{d\rho\}\leq r_i\}
  \end{equation} 
  Using the second consequence of~\eqref{convergents}, $d_{i+1}$
  is thus a term of $(K_n)_{n\in\n}$.  Similarly, we next get $\{d_{i+2}\rho\}\leq
  r_{i+1}$.  Since $r_i \geq \{d_{i+1}\rho\}$ from~\eqref{di}, observe that
    $r_{i+1} = \{\log_p m - b_i\rho - d_{i+1}\rho \} = r_i - \{d_{i+1}\rho\} <
    r_i$, so that $d_{i+2}\geq d_{i+1}$. Thusly the first claims
    regarding the terms of $(d_{i+1})_{i\in\n}$ are proved.
  \medskip

  Let us now prove the properties in (i).  Assume $d_{i+1}=k_{2s,t}$.
  Since $\{k_{2s,t}\rho\}= \eps_{2s}-t\eps_{2s+1}$, we get using~\eqref{di}
  again
  $$
  \eps_{2s}-t\eps_{2s+1}\leq r_i < \eps_{2s}-(t-1)\eps_{2s+1},  
  $$
  thus $t \geq (\eps_{2s} - r_i)/\eps_{2s+1} > t-1$, hence~\eqref{ts}.  It
  follows that $r_{i+1} = r_i-\{k_{2s,t}\rho\} < \eps_{2s+1}$, and since
  the minimum value of $\{k_{2s,t'}\rho\}$, reached for $t'=a_{2s+2}-1$, is
  $\eps_{2s+2}+\eps_{2s+1}$, we must have $d_{i+2}>k_{2s,a_{2s+2}-1}$, thus
  $d_{i+2}\geq k_{2s+2}$.
  \medskip

  Turning to (ii), assume that $d_{i+1}=k_{2s}$. Hence $r_i\geq
  \eps_{2s}$, so that $n_s$ given in~\eqref{ns} is positive.  If $n_s>1$ we get
  $r_{i+1} = r_i-\eps_{2s}\geq \eps_{2s}$, thus $d_{i+2}=k_{2s}$.  By iterating,
  it follows that $r_{i+j-1}\geq\eps_{2s}$ and $d_{i+j}=k_{2s}$ for each $j\leq
  n_s$, and that $r_{i+n_s}<\eps_{2s}$.  Thus $d_{i+n_s+1}>k_{2s}$.

  Finally assume without loss of generality that either $i=0$ or
  $d_{i}<d_{i+1}$.  Thus $s=0$ implies $i=0$.  Since $r_0=\{\log_p m\}<1$ and
  $\eps_0=\{\rho\}$, $n_0 \leq \lfloor 1/\{\rho\} \rfloor = a_1$.  If $s>0$,
  since $d_{i+1}> k_{2s-2,a_{2s}-1}$ we get, using~\eqref{eq:eps}
  $$
  r_i<\eps_{2s}+\eps_{2s-1}=(a_{2s+1}+1)\eps_{2s}+\eps_{2s+1}, 
  $$
  thus $n_s\leq a_{2s+1}+1$.  Finally, when $d_{i} = k_{2s-2,t}$ with
  $0<t<a_{2s}$,~\eqref{ts} shows that $r_i < \eps_{2s-1}$, which is tighter and
  yields $n_s \leq a_{2s+1}$ in the same way.
\end{proof}

We now describe how Theorem~\ref{th:seqd} can be turned into an algorithm
  that computes sequence $(b_{i})_{i\in\n}$.  As seen in Lemma~\ref{lemm:basis}
  the same algorithm can be used to compute either $z_m$ or $y_m$. For that
  purpose, we use an additional input parameter, denoted $B_m$, standing for
  $\left\lfloor \log_q m \right\rfloor$ or $m_\ell$ respectively. The
algorithm works as follows. Starting from values $s = 0$, $i_0=0$, $b_0=0$,
$r_0=\{\log_p m\}$, iterate:
\begin{enumerate}
    \setlength{\itemsep}{2ex}
\item \textbf{if} $(r_{2s}\geq\eps_{2s})$ \textbf{then}\\[-1ex]
  \begin{enumerate} 
    \setlength{\itemsep}{2ex}
  \item $n_s = \left\lfloor r_{2s}/\eps_{2s} \right\rfloor$; \quad 
    $r_{2s+1} = r_{2s} - n_s\eps_{2s}$;  \quad
    $i_{2s+1} = i_{2s} + n_s$;
  \item \textbf{for} $i_{2s}<i\leq i_{2s+1}$ \textbf{do} \quad $d_i = k_{2s}$;
    \quad $b_i=b_{i-1}+k_{2s}$;
  \item \textbf{if} $(b_{i_{2s+1}}>B_m)$ \textbf{then}
    \textbf{return} $b_{i_{2s}} + \left\lfloor (B_m - b_{i_{2s}})/k_{2s}
    \right\rfloor k_{2s}$;\\[-1ex]
  \end{enumerate}
  \textbf{else} $r_{2s+1} = r_{2s}$; \quad $i_{2s+1}=i_{2s}$;
\item \textbf{if} $(r_{2s+1} \geq \eps_{2s} - \eps_{2s+1})$ \textbf{then}\\[-1ex]
  \begin{enumerate} 
    \setlength{\itemsep}{2ex}
  \item $t_s = \left\lceil (\eps_{2s} - r_{2s+1})/\eps_{2s+1} \right\rceil$;
  \item
    $r_{2s+2} = r_{2s+1} - \eps_{2s} + t_s\eps_{2s+1}$; \quad
    $i_{2s+2} = i_{2s+1}+1$;
  \item $d_{i_{2s+2}} = k_{2s,t_s}$; \quad $b_{i_{2s+2}} = b_{i_{2s+1}} + k_{2s,t_s}$;
  \item \textbf{if} $(b_{i_{2s+2}} > B_m)$ \textbf{then return} $b_{i_{2s+1}}$\\[-1ex]
  \end{enumerate}
  \textbf{else} $r_{2s+2} = r_{2s+1}$; \quad $i_{2s+2}=i_{2s+1}$;
\end{enumerate}

\begin{coro}
The above algorithm requires at most $2+\lfloor\log_2\log_q m\rfloor$ iterations.
\end{coro}

\begin{proof}
  Let $z_m=\zeta(b_I)$ and let $d_I=k_{2S,t}$, with $S\geq 0$ and either $t=0$
  or $t=t_S$.  Then $k_{2S}\leq b_I\leq \log_q m$.  But relations $k_{i+3}=
  (a_{i+2}a_{i+1}+1)k_{i+1}+k_i$ and $k_0=1$ imply that $k_{2i}\geq 2^i$, with
  equality only if $i\leq 1$.  Thus $S\leq \log_2\log_q m$.  Finally, checking
  the stopping condition requires $n_{S+1}$ be computed in the case $t=t_S$.
\end{proof} 

Observe that the precision required is about $\log_2\log_q m$ bits for the
floating point calculations with fractional parts. Alternatively, the
computation may be carried-out with integers, by approximating $\rho$ by the
convergent $H/K$, where $K$ is the greatest element in $(K_n)_{n\in\n}$ not
exceeding $\log_q m$, and by performing the operations modulo $K$.

The above algorithm requires $m_\ell$ to be known in order to output $y_m$.  If
we know the largest $n$ such that $K_n+\lfloor\alpha(K_n)\rfloor/\rho\leq\log_p
m$, we simply have $m_\ell=\lfloor\log_p
m-\lfloor\alpha(K_n)\rfloor/\rho\rfloor$.  Indeed,
$\ell_b=\ell_{K_n}=\lfloor\alpha(K_n)\rfloor$ for all $b\in[K_n, K_{n+1})$.  In
order to compute this $K_n$, it suffices to

1: compute the largest $s$ such that $k_{2s}+\lfloor\alpha(k_{2s})\rfloor/\rho\leq\log_p m$, 

2: compute the largest $t$ such that $k_{2s,t}+\lfloor\alpha(k_{2s,t})\rfloor/\rho\leq\log_p m$. 

\noindent Task 1 requires at most $2+\log_2\log_q m$ steps, each step
essentially consisting in computing next values of $k_{2s}$ and
$\alpha(k_{2s})$.  Task 2 may be done by using a binary search of $t$ in
$[0,a_{2s+2})$, which requires $\log_2 a_{2s+2}$ similar steps.  Whereas we are
sure that $a_{2s}<\log_q m$ because $a_{2s}<k_{2s}$, it may happen that
$a_{2s+1}$ or $a_{2s+2}$ exceed $\log_q m$.  In the first case we conclude that
$t=0$, since $\log_q m<k_{2s,1}=k_{2s}+k_{2s+1}$.  In the second case, we may
simply use a binary search of $t$ in $[0,\lfloor (\log_q
m-k_{2s})/k_{2s+1}\rfloor]$ because $k_{2s,t}\leq \log_q m$ must hold.  We have
thus established:

\begin{prop}
  Computing $m_\ell$ can be done by computing $K+\lfloor\alpha(K)\rfloor/\rho$
  for at most $2+2\lfloor\log_2\log_q m\rfloor$ values of $K$ in the sequence
  $(K_n)$.
\end{prop}

The final remaining problem is that computing $\alpha(b)$ may be expensive. 
Recall that 
\begin{align}\label{alphabis}
\alpha(b)= \log_p \frac{q^{b}-1}{q^b-p^{\lfloor b\rho\rfloor}}+ \log_p \frac{q-p}{q-1}
\end{align}
We next show that a fitting approximation of $\alpha$ is given by 
\begin{align}\label{alpha+}
\alpha^+(b)=\log_p\frac{q-p}{(q-1)\log p}+\log_p\frac 1{\{b\rho\}}+\frac12\{b\rho\}. 
\end{align}

\begin{prop}
For all $n\in\n$,  we have
$$
0 < \alpha^+(K_n)-\alpha(K_n) < \frac{\log p}{6}\{K_n\rho\}^2+\frac{1}{q^{K_n}-1}
$$ 
\end{prop}
\begin{proof}
Let $\delta_n=\alpha^+(K_n)-\alpha(K_n)$ and $u_n=\frac 12\{K_n\rho\}\log p$.
According to \eqref{alphabis} and \eqref{alpha+},  
$$
\delta_n=\log_p\frac{\sinh u_n}{u_n}-\log_p(1-q^{-K_n})
$$
Observe that $\delta_n>0$ because $\sinh u_n> u_n>0$ and $0<q^{-K_n}<1$. 
Furthermore, $-\log(1-x)<x/(1-x)$ for $0<x<1$, thus $-\log_p(1-q^{-K_n})<1/(q^{K_n}-1)$. 
Finally, $(1-e^{-x})/x<1-x/2+x^2/6$ for $0<x$, so that 
$$
\log\frac{\sinh u_n}{u_n}=u_n+\log\frac{1-e^{-2u_n}}{2u_n}<\frac{2u_n^2}3
$$
Hence the claimed inequalities. 
\end{proof}

\section{Concluding remark}
\label{sec:conclusion}

When $m = q^N$ for some positive integer $N$, applying Theorem~\ref{th:seqd} for
the search of $z_m$ provides one with a finite representation $N$ by a finite
sum of terms of $(K_n)_{n\in\n}$.  This representation is similar in spirit to
Ostrovski's number system~\cite{Ostrowski22, Berthe01a}.  For example, take
$\rho=\log_2 3$ whose partial quotients start with $[1,1,1,2,2,3,1,...]$.  In
this case, sequence $(K_n)_{n\in\n}$ starts with $(1,2,7,12,53,...)$ and
sequence $(k_n)_{n\in\n}$ starts with $(1,1,2,5,12,41,53,...)$.  Therefore, we
get a representation of 6 given by $6 = 3K_1 = 3k_2$, whereas
$6 = k_1+k_3$ in Ostrovski's representation.  Studying this novel representation
is beyond the scope of this paper and should be the topic of future work. 

To conclude, the authors wish to thank an anonymous referee for her/his stimulating comments, 
which, in particular, lead them to discover and prove Theorem~\ref{th:seqd}.  



\providecommand{\bysame}{\leavevmode\hbox to3em{\hrulefill}\thinspace}
\providecommand{\MR}{\relax\ifhmode\unskip\space\fi MR }
\providecommand{\MRhref}[2]{%
  \href{http://www.ams.org/mathscinet-getitem?mr=#1}{#2}
}
\providecommand{\href}[2]{#2}

\end{document}